\documentclass[11pt,reqno]{amsart}
\usepackage{amsfonts,amssymb,amsmath,amsopn,amsthm,graphicx}
\usepackage{amsxtra, mathrsfs}
\usepackage{colordvi}
\usepackage[usenames,dvipsnames]{color}
\usepackage{amsfonts,amssymb,amsbsy,amsmath,amsthm,dsfont}

\usepackage{amsmath,amssymb,amsfonts,amscd,hyperref,color}
\usepackage[latin1]{inputenc}
\usepackage{verbatim}
 \numberwithin{equation}{section}
\newcommand{\Hmm}[1]{\leavevmode{\marginpar{\tiny%
			$\hbox to 0mm{\hspace*{-0.5mm}$\leftarrow$\hss}%
			\vcenter{\vrule depth 0.1mm height 0.1mm width \the\marginparwidth}%
			\hbox to
			0mm{\hss$\rightarrow$\hspace*{-0.5mm}}$\\\relax\raggedright #1}}}
\newcommand{\K}{\mathcal{K}}
\newcommand{\loc}{{\rm loc}}

\newcommand{\N}{\mathbb{N}}

\newcommand{\R}{\mathbb{R}}

\newcommand{\id}{\mathds{1}}

\newtheorem{theorem}{Theorem}[section]
\newtheorem{corollary}[theorem]{Corollary}
\newtheorem{lemma}[theorem]{Lemma}

\newtheorem{definition}[theorem]{Definition}
\newtheorem{remark}[theorem]{Remark}

\theoremstyle{definition}

\newtheorem{defin}[theorem]{Definition}

\newcommand{\dx}{\,\mathrm{d}x}
\newcommand{\dy}{\,\mathrm{d}y}
\newcommand{\dz}{\,\mathrm{d}z}
\newcommand{\dt}{\,\mathrm{d}t}

\newcommand{\dr}{\,\mathrm{d}r}

\newcommand{\be}{\begin{equation}}
\newcommand{\ee}{\end{equation}}
\newcommand{\bea}{\begin{eqnarray}}
\newcommand{\eea}{\end{eqnarray}}
\newcommand{\bean}{\begin{eqnarray*}}
	\newcommand{\eean}{\end{eqnarray*}}

\newcommand{\Real}{\mathbb{R}}

\newcommand{\opname}[1]{\mbox{\rm #1}\,}
\newcommand{\supp}{\opname{supp}}

\newlength{\wex}  \newlength{\hex}
\def\ga{\alpha}            \def\gg{\gamma}
       \def\gd{\delta}      
                         
\def\gf{\phi}       \def\vgf{\varphi}    
            \def\gl{\lambda}

      \def\gw{\omega}
                
     \def\Gd{\Delta}

\def\Gw{\Omega}              
\begin{document}
	
	\title[Decay of Hardy-weights on exterior domains]{On minimal decay at infinity of Hardy-weights}
	
	%\iffalse
	
	\author {Hynek Kova\v{r}\'{\i}k}
	
	\address {Hynek Kova\v{r}\'{\i}k, DIGICAM, Sectioned di Matematica \\ Universit\`a degli Saudi di Brescia \\ Via Braze, \\38 - 25123, Brescia, Italy}
	
	\email {hynek.kovarik@unibs.it}
	
	\author{Yehuda Pinchover}
	
	\address{Yehuda Pinchover,
		Department of Mathematics, Technion - Israel Institute of
		Technology,   Haifa, Israel}
	
	\email{pincho@technion.ac.il}

	%\fi
	
	%%%%%%%%%%%%%%

	\begin{abstract}
	We study the behaviour of Hardy-weights for a class of variational quasi\-linear elliptic operators of $p$-Laplacian type. In particular, we
	obtain necessary sharp decay conditions at infinity on the Hardy-weights in terms of their integrability with respect to certain integral weights. 
	Some of the results are extended also to nonsymmetric linear elliptic operators.
	Applications to various examples are discussed as well.  
	
	\medskip
		
		\noindent  2000  \! {\em Mathematics  Subject  Classification.}
		Primary  \! 47J20; Secondary  35B09, 35J08, 35J20, 49J40.\\[1mm]
	\end{abstract}
	
	\maketitle

	% {\bf AMS Mathematics Subject Classification:}  47F05, 49R05 \\

	%{\bf  Keywords:} Robin Laplacian, lowest eigenvalue, convex domains \\

	%%%%%%%%%%%%%%%%%%%%%%%%%%%%%%%%%%%%%%%%%%
	
	%%%%%%%%%%%%%%%%%%%%%%%%%%%%%%%%%%%%%%%%%%
\section{\bf Introduction}
The problem of finding a function $W\gneqq0$ such that a given nonnegative operator $L$ in a domain $\Omega\subset \R^N$ satisfies, in certain sense, the inequality 
\begin{equation} \label{hardy-intro}
L \, \geq\, W
\end{equation}
has been intensively studied in the past decades. For a detailed analysis of these so-called Hardy-type inequalities we refer to the monographs \cite{bel,kps,ok} and 
references therein. The function $W$ is usually called a Hardy-weight for the operator $L$. 

\smallskip

The aim of the present paper is to study  the  `decay' at infinity of Hardy-weights for  a class of variational quasilinear operators, the so-called $(p,A)-$Laplacians with a potential term, in external domains, see Section \ref{s-prelim} for a detailed definition. For practical  purposes, and also from the theoretical point of view, it is very natural to address the question of `how large' the weight function $W$ might be. One would of course like to make $W$ 
as large as possible in order to optimize inequality \eqref{hardy-intro}  (see \cite{dfp,dp}). However, there are certain constrains, depending on $L$ and $\Omega$, which have to 
be respected. 

 \smallskip
 
In this note we study one of such constrains, namely the behaviour of $W$ at infinity. Roughly speaking we show that the Hardy-weights cannot decay too slowly at infinity. More precisely, we establish necessary decay conditions on $W$ in terms of $L^1$ integrability of $W$ (at infinity) with respect to integral weights which are related to positive solutions of the equation $L u=0$. This is done in Section~\ref{s-critical} for critical operators (see Theorem~\ref{prop-L1}), and in Section~\ref{s-subcritical} for subcritical operators (see Theorem~\ref{prop-L2}). In Section~\ref{sec-cr-lin} we show that the results can be extended also to a certain class of linear nonsymmetric operators  (see, Theorem~\ref{thm-nonsym}). 

In Section~\ref{s-prelim} we introduce the necessary notation and recall some results previously obtained in the literature on some of which we rely in the proofs of our main theorems. The latter are formulated and proved in Section~\ref{s-main}. In the closing Section~\ref{s-examples}, we illustrate the sharpness of our decay conditions by several examples.

%%%%%%%%%%%%%%%%%%%%%%%%%%%%%%%%%%%%%%%%%%%%%
	\section{\bf Preliminaries} \label{s-prelim}
\subsection{Notation} \label{ss-notation}
	Let $\Gw$ be a domain in $\R^N$, $1<p<\infty$, $N\geq 2$.  Throughout the paper we use the following notation.
	\begin{itemize}
		\item We denote  by $ \bar{\infty} $ the ideal point which is added to $ \Gw $ to obtain the one-point compactification of $\Gw$. 
		
		\item  We write $\Gw_1 \Subset \Gw_2$ if the set $\Gw_2$ is open in $\Gw$, the set $\overline{\Gw_1}$ is
		compact and $\overline{\Gw_1} \subset \Gw_2$. 
		
		\item Let $g_1,g_2$ be two positive functions defined in a domain $D$. We say that $g_1$ is {\em equivalent} to $g_2$ in $D$ (and use the notation $g_1\asymp g_2$ in
		$D$) if there exists a positive constant $C$ such
		that
		$$C^{-1}g_{2}(x)\leq g_{1}(x) \leq Cg_{2}(x) \qquad \mbox{ for all } x\in D.$$
		
		\item The open ball of radius $r>0$ centered at $y$ is denoted by $$B_r(y) := \{x\in\R^N:\, |x-y| <r \}.$$
		
		\item We denote by  $\mathrm{diam} (\gw)$ the diameter of the set $\gw \subset \R^N$.
	\end{itemize}
	
\subsection{The quasilinear case}\label{subsec-quasi}
	Let $\Gw$ be a domain in $\R^N$, $1<p<\infty$, $N\geq 2$.  In this note we consider the quasilinear operator  $ L(A,V) $ acting on $W^{1,p}_\loc(\Gw)$ of the form 
	\begin{equation}  \label{L-op}
	L(A,V)  u  :=- \nabla \cdot \big(|\nabla u|_A^{p-2} A\,  \nabla u\big) +V |u|^{p-2}\, u ,
	\end{equation}
	and the following associated functional defined on $C_c^\infty(\Gw)$
	\begin{equation}\label{functional}
	Q_{A,V}[\varphi]:= 	\int_{\tilde{\Omega}}(|\nabla \varphi|^p_A+ V|\varphi|^p)\dx,
	\end{equation} 
	where $A:\Gw\to \R^{N}\times \R^{N}$ is a symmetric matrix, $V:\Gw \to \R$ is a potential function, and  
	$$
	|\xi |_{A(x)} : = \langle \xi, A(x) \, \xi \rangle^{1/2} \qquad x\in \Gw, \, \xi\in \R^N. 
	$$
	We assume that $A$ is locally bounded and locally uniformly positive definite, i.e., for any $K\Subset \Gw$  there exists  $\Lambda_K>0$  such that
	\begin{equation} \label{eq-ell}
	\Lambda_K^{-1}\, |\xi|^2 \leq |\xi |_{A(x)}^2   \leq  \Lambda_K\, |\xi|^2 \qquad \forall\, x\in K, \ \ \forall\, \xi\in\R^N,
	\end{equation} 
	and $V$ belongs to a  certain local Morrey space $M^q_{\rm loc}(p;\Gw)$  which is defined below (see \cite{pp}).
	\begin{defin} \label{def-morrey}
		Let $q\in [1,\infty]$ and let $\omega \Subset \R^N$. Given a measurable function $f:\omega\to \R$, we set 
		\begin{equation} 
		\|f\|_{M^q(\omega)} := \sup_{y\in\omega} \, \sup_{r< {\rm diam} (\omega)}\left\{r^{-N(q-1)/q} \int_{B_r(y) \cap \omega} |f|\dx\right\}.
		\end{equation}
		We write $f\in M_{\rm loc}^q(\Gw)$ if for any $\omega\Subset\Gw$ it holds $\|f\|_{M^q(\omega)} < \infty$.  We then define 
		
		\begin{equation}
		M^q_{\rm loc} (p; \Gw) : = \left\{
		\begin{array}{l@{\quad}l}
		M^q_\loc(\Omega) \ \  \text{with} \ q >\frac Np & \quad \text{if} \ \ p<N  , \\
		& \\
		L^1_{\rm loc}(\Gw)  & \quad \text{if} \ \ p >N,
		\end{array}
		\right.  
		\end{equation}
		and for $p=N$ we write $f	\in M^q_{\rm loc} (N;\Gw)$, if for some $q>N$ and any $\omega\Subset\Gw$ it holds
		\begin{equation}
		\|f\|_{M^q(N,\omega)} := \sup_{y\in\omega} \, \sup_{r< {\rm diam} (\omega)}\left\{ \log^{\frac{q(N-1)}{N}}\left(\frac{ {\rm diam} (\omega)} {r}\right) \int_{B_r(y) \cap \omega} |f|\dx\right\} < \infty.
		\end{equation}
	\end{defin}
	We are now ready to introduce our regularity hypotheses on the coefficients of the operator $L(A,V)$. Throughout the paper we assume that
	\[\tag{H0}
	\mbox{the matrix }A\mbox{ satisfies } \eqref{eq-ell},\mbox{ and the potential }V\in M^q_{\rm loc}(p;\Gw).
	\]
	In the case $1<p<2$, we make the following stronger hypothesis:
	
	\[\tag{H1}
	A\in C^{0,\gamma}_{\rm loc}(\Gw;\R^{N^2})\mbox{ satisfies } \eqref{eq-ell},\mbox{ and }V\in M^q_{\rm loc}(\Gw),\mbox{ where }q> N.
	\]

	\begin{defin}\label{def:critical}   The functional $Q_{A,V}$ is said to be
	\begin{enumerate}
	 \item {\em nonnegative} in $\Gw$ (in short,   $Q_{A,V}\geq 0$ in $\Gw$) if 
		\begin{equation}\label{ineq:nonnegative}
		Q_{A,V}[\vgf]:= 	\int_{\tilde{\Omega}}(|\nabla \vgf|^p_A+ V|\vgf|^p)\dx
		\geq
		0   \qquad \mbox{for all } \vgf\in C_c^\infty(\Gw).
		\end{equation}	
		
	\item  {\em subcritical} in $\Gw$ if there exists a nonzero nonnegative weight function $W\in  M^q_{\rm loc}(p;\Gw)$, called a {\em Hardy-weight}, such that
		\begin{equation}\label{ineq:subcriticality}
		Q_{A,V}[\vgf]
		\geq
		\int_\Gw W|\vgf|^p\dx   \qquad \mbox{for all } \vgf\in C_c^\infty(\Gw).
		\end{equation}
	\item {\em critical}  in $\Gw$ if $Q_{A,V}\geq 0$ in $\Gw$, but $Q_{A,V}$ does not admit any Hardy-weight.  
		
         \item  {\em supercritical} in $\Gw$ if $Q_{A,V}\not \geq 0$  in $\Gw$ (that is, there exists $\vgf\in C_c^\infty(\Gw)$ such that $Q_{A,V}[\vgf]<0$).
         \end{enumerate}
	\end{defin}
	
\begin{defin}
We say that the {\em operator $L(A,V)$ is nonnegative in $\Gw$} if the equation $L(A,V)w=0$ in $\Gw$ admits a positive weak supersolution  $\tilde u\in W_{\loc}^{1,p}(\Gw)$. 
\end{defin}

First, we recall the following Allegretto-Piepenbrink-type theorem.
	\begin{theorem}[{Allegretto-Piepenbrink-type theorem \cite[Theorem~4.3]{pp}}]\label{theorem:AP} Suppose  $A$ and $V$ satisfy hypothesis {\em(H0)}. Then the following assertions are equivalent:
		\begin{itemize}
			\item  The functional $Q_{A,V}$ is nonnegative on $C_c^\infty(\Gw)$.
			
			\item The equation $L(A,V)w=0$ in $\Gw$ admits a positive solution  $v\in W_{\loc}^{1,p}(\Gw)$.
			\item The equation $L(A,V)w=0$ in $\Gw$ admits a positive supersolution  $\tilde v\in W_{\loc}^{1,p}(\Gw)$.
		\end{itemize}
	\end{theorem}
	\begin{defin}\label{def:nullseq} A sequence $\{\vgf_k\}\subset C_c^\infty(\Gw)$ is called a {\em null-sequence} with respect to the nonnegative functional $Q_{A,V}$ in $\Gw$ if
		\vspace{.5em}
		
		a) $   \vgf_k\geq0$ for all $k\in\mathbb{N}$,
		\vspace{.5em}
		
		b) there exists a fixed open set $K\Subset\Gw$ such that $\|\vgf_k\|_{L^p(K)}=1$ for all $k\in\mathbb{N}$,
		\vspace{.5em}
		
		c) $\displaystyle{\lim_{k\to\infty}Q_{A,V}[\vgf_k]=0}$.
		\vspace{.5em}
		
		\noindent We call a positive function $\phi\in W_{\rm loc}^{1,p}(\Gw)$ a {\em ground state} of $Q_{A,V}$ in $\Gw$ if $\phi$ is an $L_{\rm loc}^p(\Gw)$ limit of a null-sequence.
	\end{defin}

\noindent We have \cite[Theorem~4.15]{pp}:

\begin{theorem}\label{theorem:main} Suppose that $Q_{A,V}$ is nonnegative on $C_c^\infty(\Gw)$ with $A$ and $V$  satisfying hypothesis {\em(H0)} if $p\geq2,$ or {\em (H1)} if $1<p<2$. Then $Q_{A,V}$ is critical in $\Gw$ if and only if $Q_{A,V}$ admits a null-sequence.
Moreover, in this case the equation $L(A,V)u=0$ admits (up to multiplicative constant) a unique positive supersolution $\gf$. Furthermore, $\gf$ is a ground state.
\end{theorem}

If $Q_{A,V}$ is subcritical in $\Gw$, then  the set of all Hardy-weights is convex. Moreover, $W$ is an extreme point of this set if and only if $Q_{A,V-W}$ is critical. This indicates that critical Hardy-weights are rare, and in general difficult to be determined concretely. The papers \cite{dfp,dp} are devoted to the search of a class of {\em optimal} Hardy-weights, that is, Hardy-weights that are `{\em as large as possible}' in the following sense.  
	\begin{defin}\label{def_opt}
		Suppose that $Q_{A,V}\geq 0$ in $\Gw$. Assume that a nonzero nonnegative function $W$ satisfies the following Hardy-type inequality
		\begin{equation}\label{opt_hardy4}
		Q_{A,V}[\varphi] \geq \gl\int_{\Omega} W|\varphi|^p\dx \qquad \forall \varphi\in C_c^\infty(\Omega),
		\end{equation}
		with some $\gl>0$. We say that $W$ is an {\em optimal Hardy-weight} for the operator $Q_{A,V}$ in $\Gw$ if the following conditions hold true.
		\begin{itemize}
			
			\item (Criticality) The functional $Q_{A,V-W}$ is {\em critical in} $\Gw$. In particular, $Q_{A,V- W}$ admits a ground state $\gf$ in $\Omega$.
			
			\item (Null-criticality) The functional $Q_{A,V\!-\!W}$ is {\em null-critical in} $\Gw$ with respect to $W$, that is, $\gf\!\not\in\! L^p(\Gw,W\!\dx)$. 
		
		\item  (Optimality at infinity) $\gl=1$ is also the best constant for inequality \eqref{opt_hardy4} restricted to functions $\vgf$ that are compactly supported  in any fixed neighborhood of infinity in $\Gw$.
	\end{itemize}
	\end{defin}
	
	\medskip 
	
	\noindent For the $p$-Laplacian in `exterior' domains we have
	
	\begin{theorem}[{\cite[Theorem~6.1]{dp}}]\label{thm_opt_hardym}
		Let $\Gw$ be a $C^{1,\alpha}$ domain (not necessarily bounded), where $0<\alpha\leq 1$. Let $U\Subset \Omega$ be an open $C^{1,\alpha}$ subdomain of $\Omega$, and consider $\tilde{\Gw}:=\Gw\setminus U$. Denote by $\bar{\infty}$ the infinity in $\Omega$, and assume that $-\Delta_p$ admits a nonconstant positive $p$-harmonic function $\mathcal{G}$ in $\tilde{\Omega}$ satisfying the following conditions
		\begin{equation}\label{assumpt_7m}
		\lim_{x\to \partial U}\mathcal{G}(x)=\gg_1, \; \lim_{x\to\bar{\infty}}\mathcal{G}(x)= \gg_2, \end{equation}
		where $\gg_1\neq \gg_2$, and  $0\leq \gg_1,\, \gg_2 \leq \infty$. Denote $$m:=\min\{\gg_1,\gg_2\}, \qquad M:=\max\{\gg_1,\gg_2\}.$$
		
		\noindent Define positive functions $v_1$ and  $v$, and a nonnegative weight $W$ on $\tilde{\Omega}$ as follows:
		
		\medskip
		
		(a) If $M<\infty$, assume further that either $m=0$ or $p\geq 2$, and let
		$$v_1:=(\mathcal{G}-m)(M-\mathcal{G}), \qquad  v:=v_1^{(p-1)/p}=[(\mathcal{G}-m)(M-\mathcal{G})]^{(p-1)/p},$$
		and
		\be\label{wgg1gg21m}
		W:=\left(\frac{p-1}{p}\right)^p\left|\frac{\nabla \mathcal{G}}{v_1}\right|^p|m+M-2\mathcal{G}|^{p-2}\left[2(p-2)v_1+(M-m)^2\right].
		\ee
		
		(b) If $M=\infty$, define
		$$ v_1:=(\mathcal{G}-m), \qquad  v:=v_1^{(p-1)/p}=(\mathcal{G}-m)^{(p-1)/p},$$
		and
		\be\label{wgg1gg21m2}
		W:=\left(\frac{p-1}{p}\right)^p\left|\frac{\nabla \mathcal{G}}{v_1}\right|^p.
		\ee
		
		\medskip
		
		\noindent Then the following Hardy-type inequality holds true
		\begin{equation}\label{opt_hardym}
		\int_{\tilde{\Omega}}|\nabla \varphi|^p\dx \geq \int_{\tilde{\Omega}} W|\varphi|^p\dx \qquad \forall \varphi\in C_c^\infty(\tilde{\Omega}),
		\end{equation}
		and $W$ is an {\em optimal} Hardy-weight for $-\Gd_p$ in $\tilde{\Gw}$.
		
		Moreover, up to a multiplicative constant, $v$ is the unique positive supersolution of the equation $L(\id,-W)w=0$ in $\tilde{\Gw}$.
	\end{theorem}
	%%%%%%%%%%%%%
	
	\subsection{The linear case} \label{sec-lin}
	In the same token, we consider also  {\em linear} (not necessarily symmetric) second-order  elliptic operators $P$ with {\em real} coefficients in divergence form:
	\begin{equation} \label{operator}
	Pu:=-\nabla \!\cdot\!\left[A(x)\nabla u +  u\tilde{b}(x) \right]  +
	b(x)\!\cdot\!\nabla u   +c(x)u \qquad x\in \Gw.
	\end{equation}
	We assume that $A$ satisfied hypothesis (H0),  and $b$ and $\tilde b$ are measurable vector fields in
	$\Gw$ of class $L^q_\loc(\Gw,\R^N)$  and $c$ is a measurable function in $\Gw$ of class $L^{q/2}_\loc(\Gw,\R)$ for some $q> N$.
	By a {\em solution} $v$ of the equation $Pu = 0$ in $\Gw$,  we mean that $v \in W^{1,2}_{\loc}(\Gw)$ and satisfies the equation in the {\em weak} sense. Subsolutions and supersolutions are defined similarly.

	We denote by $P^\star$ the {\em formal adjoint} operator of $P$ on its natural space $L^2(\Gw,\!\dx)$. If $b = \tilde{b}$, then the operator $P$ is 
	{\em symmetric} in the space $L^2(\Gw, \!\dx)$, and we call this setting the {\em linear symmetric case}. 
	We note that if $P$ is symmetric and $b$ is smooth enough, then $P$ is in fact a Schr\"odinger-type operator of the form $Pu =  - \nabla \cdot \big(A \nabla u \big) + \tilde c u$, where $\tilde c=c-\nabla \cdot b$.
	\begin{remark}{\em 
		Our results hold true also when $P$ is of the form 
		\begin{equation} \label{op-P}
		P = -\sum_{i,j=1}^N a_{ij}(x) \, \partial_i \, \partial_j +\sum_{i01}^N b_i(x)\, \partial_i +c(x).
		\end{equation} 
		In this case we should assume that the coefficients $a_{ij}, b_i$ and $c$ are H\"older continuous and that the quadratic form
		$$
		\sum_{i,j=1}^N a_{ij}(x)\, \xi_i\, \xi_j , \qquad \xi\in\R^N
		$$
		is positive definite for all $x\in\Omega$. In this framework we consider classical solutions and supersolutions.
	} 
	\end{remark}
	
\begin{defin}  We say that the operator $P$ is 
\begin{enumerate}
\item {\em nonnegative in} $\Gw$ (and we write $P\geq 0$ in $\Gw$) if the equation $Pu=0$ in $\Gw$ admits a positive (super)solution. 
		
\medskip 
		
\item {\em subcritical} in $\Omega$ if there exists  a function $W\gneqq0$	in $\Gw$ such that $P-W\geq 0$ in $\Gw$. Such a weight $W$ is called a {\em Hardy-weight} for the operator $P$ in $\Gw$. If $P\geq 0$ in $\Gw$,  but $P$ does not admit any Hardy-weight, then $P$ is said to be {\em critical} in $\Gw$. 
		
\end{enumerate}
\end{defin}
For more details concerning {\em criticality theory}, see for example the review article \cite{p07} and references therein. In particular, we need the following result. 
\begin{lemma}\label{lem-crit} The following claims hold true;
\begin{enumerate}
\item The operator $P$ is critical in $\Gw$ if and only if $P^*$ is critical in $\Gw$.
		
\item The operator $P$ is critical in $\Gw$ if and only if the equation $Pu=0$ in $\Gw$ admits (up to a multiplicative constant) a unique  positive supersolution called the Agmon ground state (or in short ground state). 
		
\item The operator $P$ is subcritical in $\Gw$ if and only if $P$ admits a  {\em positive minimal Green function} $G_P^\Gw(x,y)$ in $\Gw$.  
\end{enumerate}
 \end{lemma}

\begin{definition}\label{groundstate}{\em
		Let $K \Subset \Gw$. We say that a positive solution $u$ of the equation $Pw=0$ in  $\Gw \setminus K$ is a {\em positive solution of the
			operator $P$ of minimal growth in a neighborhood of infinity in $\Gw$}, if for
		any compact set $K \Subset K_{1} \Subset  \Gw$ with a smooth boundary and any positive supersolution $v$ of the equation $Pw=0$
		in $ \Gw \setminus K_{1}$,   $v\in C(( \Gw \setminus K_{1})\cup \partial K_{1})$, the inequality
		$u \leq v$ on $\partial K_{1}$ implies that $u \leq v$ in $ \Gw \setminus K_{1}$.
	}
\end{definition}

\begin{remark}\label{rem-min-gr}{\em Note that 

\begin{enumerate}		
\item If $P$ is subcritical in $\Gw$, then  for any fixed $y\in \Gw$,  the positive minimal Green function $G_P^\Gw(\cdot,y)$ is a positive solution of the
		equation $Pu=0$ in $\Gw\setminus \{y\}$ of minimal growth in a neighborhood of infinity in $\Gw$. 
		
\item On the other hand, in the  critical case, the ground state of $P$ in $\Gw$ is a positive solution of the
equation $Pu=0$ in $\Gw$ of minimal growth in a neighborhood of infinity in $\Gw$.
\end{enumerate}	}
\end{remark}
	\begin{defin}\label{def-opt}
		A Hardy-weight $W$ for a subcritical operator $P$ in $\Gw$ is said to be {\em optimal} if the following three properties hold:
		\begin{itemize}
			\item (Criticality) $ P-W$ is critical in $\Gw$. Denote by $\phi$ and $\phi^\star$ the ground states of $P$ and $P^\star$, respectively.
			\item (Null-criticality) $W\not\in L^1(\Gw,\phi\phi^\star\! \dx )$. 
			\item (Optimality at infinity) For any $\gl>1$ the operator $P-\gl W\not \geq 0$ in any neighborhood of infinity in $\Gw$
		\end{itemize}
	\end{defin}
	The following theorem is a version of \cite[Theorem~4.12]{dfp} (cf.  the discussion therein); we omit its proof since it can be obtained by a slight modification of the proof in  \cite{dfp}.
	\begin{theorem}\label{thm_nsa}
		Let $P$ be a subcritical operator in $\Gw$, and let $0 \lneqq\vgf\in C_c^\infty(\Gw)$. Consider the Green potential 
		$$G_\vgf(x):=\int_\Gw G_P^\Gw(x,y)\vgf(y)\dy,$$ 
		where  $G_P^\Gw$ is the minimal positive Green function. Let $u$ be a positive solution of the equation $Pu=0$ in $\Gw$ satisfying
		\begin{equation}\label{u1u0a}
		\lim_{x \to \bar\infty} \frac{G_\vgf(x)}{u(x)}=0,
		\end{equation}
		where $\bar\infty$ is the ideal point in the one-point compactification of $\Gw$. Consider the positive supersolution
		$$
		v := \sqrt{G_\vgf u}
		$$
		of the operator $P$ in $\Gw$. Then the associated Hardy-weight
		\begin{equation}\label{eq_W}
		W:=\frac{Pv}{v}\, ,
		\end{equation}
		is an optimal Hardy-weight with respect to $P$ in $\Gw$. Moreover,
		$$W=\frac{1}{4}\left|\nabla \log\left(\frac{G}{u}\right)\right|_A^2 \qquad \mbox{in } \Gw\setminus \supp{\vgf}.$$
		%%%%%%%%%%%%%%%%%%%%%%%%%%%%%%%%
		%Assume further that $P$ is a symmetric operator and $W$ is positive in $\Gw^\star$, then the spectrum and the essential spectrum of the %Friedrichs extension of the operator $W^{-1}P$ on $L^2(\Gw^\star, W\dx)$ is equal to  $[1,\infty)$, and the corresponding Agmon metric
		%
		%$$\qquad \ds^2:= W(x)\sum_{i,j=1}^{n} a_{ij}(x)\dx^i\dx^j, \quad \mbox{where } \big[a_{ij}\big]:= \big[a^{ij}\big]^{-1}$$
		%is complete.
	\end{theorem}
	
	\medskip
	
%%%%%%%%%%%%%%%%%%%%%%%%%%%%%%%%%%%%%%%%%%%%%%%%%%%%%%%%%

\section{\bf Main results} \label{s-main}
\subsection{The quasilinear critical case} \label{s-critical}
We suppose that $L(A,V)$ is nonnegative in $\Gw$, in other words (by Theorem~\ref{theorem:AP}),
\begin{equation} \label{q-form}
Q_{A,V}[\vgf]:=\int_\Gw \big( |\nabla \vgf|_A^p + V |\vgf|^p\big)\dx  \, \geq\,  0 \qquad \forall \vgf\in C_c^\infty(\Gw). 
\end{equation} 
Let $K \Subset \Gw$ be a compact set of positive measure with smooth boundary. Then by  \cite[Proposition 4.18]{pp}, the operator $L(A,V)$ is subcritical in $K^c:=\Gw\setminus K$. Hence,  there exists a Hardy-weight $W\gneqq0$, which depends on $\Gw, K, A$ and $V$, such that a Hardy-type inequality
\begin{equation} \label{ineq-hardy}
Q_{A,V}[\vgf]=  \int_{K^c}\big( |\nabla \vgf|_A^p + V |\vgf|^p\big)\dx   \, \geq \, \int_{K^c} W\, |\vgf|^p\dx \qquad \forall \vgf\in C_c^\infty(K^c)
\end{equation} 
holds true. 

%%%%%%%%%%%%

The following theorem answers the question how large (in a neighborhood of infinity in $\Gw$)  the Hardy-weight $W$ in \eqref{ineq-hardy} might be if $L(A,V)$ is critical in
$\Gw$. 
 
\begin{theorem} \label{prop-L1}
Suppose that $L(A,V)$ is critical in $\Gw$ with $A$ and $V$ satisfying either hypothesis {\em(H0)} if $p\geq2$, or {\em (H1)} if $1<p<2$, and let $\phi\in W_{\rm loc}^{1,p}(\Gw)$ be the corresponding ground state satisfying the normalization condition $\|\gf\|_{L^p(K)}=1$. Let $K \Subset \Gw$ be a compact set of positive measure with smooth boundary.

Then  for any $W\gneqq0$ satisfying \eqref{ineq-hardy} and any compact set $\K$ such that $K\Subset \K\Subset \Gw$ we have $W\in L^1( \K^c, \phi^p \dx)$. 
\end{theorem}

\begin{proof}
Let $K\Subset \K\Subset \Gw$. In view of our ellipticity assumption \eqref{eq-ell} it is possible to choose $f\in C^1(\Gw)$ satisfying 
$$
f =  \left\{
\begin{array}{l@{\quad}l}
0 & \quad \text{in} \ \  K , \\
1 & \quad \text{in} \ \ \K^c,
\end{array}
\right.  \qquad 0 \leq f \leq 1, 
$$
and such that $|\nabla f(x) |_{A} \leq C_0$ holds for all $x\in \Gw$. Since $L(A;V)$ is critical in $\Gw$, there exists a null-sequence $\{\vgf_n\}_{n\in\N} \subset C_c^\infty(\Gw)$ such that $\vgf_n \geq 0$, $\|\vgf_n\|_{L^p(K)}=1$   for all $n\in\N$, and 
\begin{equation} \label{null-s}
 \int_\Gw \big( |\nabla \vgf_n|_A^p + V |\vgf_n|^p\big)\dx  \ \to 0 \qquad \text{as} \  \ n\to\infty. 
\end{equation}
Moreover, by density and inequality \eqref{ineq-hardy} we have
\begin{equation}  \label{hardy-2}
  \int_{K^c}\big( |\nabla (f \vgf_n)|_A^p + V |f \vgf_n|^p\big)\dx   \, \geq \, \int_{K^c} W\, |f \vgf_n|^p\dx \, \geq \, \int_{\K^c} W\, | \vgf_n|^p\dx  .
\end{equation} 
Since $(a+b)^p \leq  2^{p-1}a^p +2^{p-1} b^p$ holds for all $a,b>0$ and $p\geq 1$, we obtain the upper bound 
\begin{align*}
& \int_{K^c}\!\! \big(|\nabla (f \vgf_n)|_A^p + V |f \vgf_n|^p\big)\!\dx  = \int_{\K^c}\!\! \big( |\nabla \vgf_n|_A^p + V |\vgf_n|^p\big)\!\dx + \int_{\K\setminus K} \!\!\big(|\nabla (f \vgf_n)|_A^p + V |f \vgf_n|^p\big)\!\dx \\
& \  =\int_{\Gw}\!\! \big( |\nabla \vgf_n|_A^p + V |\vgf_n|^p\big)\dx - \int_{\K}\!\! \big( |\nabla \vgf_n|_A^p + V |\vgf_n|^p\big)\dx +\int_{\K\setminus K} \!\!\big(|\nabla (f \vgf_n)|_A^p + V |f \vgf_n|^p\big)\dx \\
& \ \leq \int_{\Gw}\!\! \big( |\nabla \vgf_n|_A^p + V |\vgf_n|^p\big)\dx +2^{p-1}\!\! \int_\K|\nabla \vgf_n|_A^p \, \dx + 2 \int_\K |V|\, | \vgf_n|^p \, \dx 
+ 2^{p-1} C_0^p \int_\K  | \vgf_n| ^p \, \dx .
\end{align*}
Hence, there exists $C_p>0$ such that 
\begin{equation}  \label{hardy-3}
\int_{\K^c} W\, |\vgf_n|^p\dx  \leq  \, \int_\Gw \big( |\nabla \vgf_n|_A^p + V |\vgf_n|^p\big)\dx + C_p \int_{\K} \big( |\nabla  \vgf_n|_A^p + |V| | \vgf_n|^p+| \vgf_n|^p\big) \dx .
\end{equation} 
On the other hand, by \cite[Theorem~4.12]{pp} it then follows that the sequence $\{\vgf_n\}$ converges in $L^p_{\rm loc}$ and almost everywhere in $\Gw$ to $\phi$ and that 
\begin{equation} \label{phi-reg}
\phi, |\nabla \phi| \in L^\infty_{\rm loc}(\Gw) \qquad \text{if} \ \ \ 1<p<2. 
\end{equation}
Moreover, if we write 
$$
\vgf_n = \psi_n\, \phi,
$$
then the sequence $\psi_n$ is bounded in $W_{\rm loc}^{1,p}(\Gw)$ and $\nabla \psi_n \to 0$ in $L^p_{\rm loc}(\Gw)$, see \cite[Proposition~4.11]{pp}.
In view of the Rellich-Kondrachov theorem,  H\"older inequality and \eqref{phi-reg} it thus follows that $\vgf_n$ is bounded in $W_{\rm loc}^{1,p}(\Gw)$. Therefore, 
\begin{equation} \label{un}
\limsup_{n\to \infty} \int_{\K} \big( |\nabla  \vgf_n|_A^p + |\vgf_n|^p\big)\dx \, < \, \infty. 
\end{equation}
This in turn implies, again by using the Rellich-Kondrachov theorem, that $\vgf_n$ is bounded in $L^r_{\rm loc}(\Gw)$ for all $r\in[p,\infty]$ if $p>N$, all $r\in[p,\infty)$ if $p=N$ and all $r\in[p,p^*]$ if $p<N$, where $ p^*=\frac{pN}{N-p}$. By assumption, we have $p q' < p^*$. Then by the  H\"older inequality  
\begin{equation} \label{Vun}
\limsup_{n\to \infty} \int_{\K} \big( |V|\, |\vgf_n|^p\big)\dx \ \leq \ \limsup_{n\to \infty}\Big( \int_{\K}  |V|^q \dx \Big)^{\frac 1q}  \Big( \int_{\K} |\vgf_n|^{pq'} \dx \Big)^{\frac{1}{q'}}  \ < \infty .
\end{equation}

\medskip

To complete the proof we note that in view of the pointwise a.e.~convergence of $\vgf_n$ to $\phi$ and the Fatou lemma it holds
$$
\liminf_{n\to\infty} \int_{\K^c} W |\vgf_n|^p\dx \ \geq \  \int_{\K^c} W  \liminf_{n\to\infty} |\vgf_n|^p\dx = \int_{\K^c} W |\phi|^p\dx\ . 
$$
This in combination with \eqref{null-s}, \eqref{hardy-3}, \eqref{un}  and  \eqref{Vun} gives
$$
\int_{\K^c} W |\phi|^p\dx \ < \infty,
$$
and the claim follows. 
\end{proof}

%%%%%%%%%%%%%%%%%%%%%%%%%%%%%%%%%%%%%%%%%%%%%%%%%%%%%%%%%%%%%
\subsection{The quasilinear subcritical case} \label{s-subcritical}
%%%%      
We have
\begin{theorem} \label{prop-L2}
	Suppose that $L(A,V)$ of the form \eqref{L-op} is subcritical in $\Gw$ with $A$ and $V$ satisfying either hypothesis {\em(H0)} if $p\geq2$, or {\em (H1)} if $1<p<2$. Let $K \Subset \Gw$ be a compact set of positive measure with smooth boundary,  and let $\phi\in W_{\rm loc}^{1,p}(\Gw\setminus K)$ be a positive solution of the equation $L(A,V)[u]=0$ in  $\Gw\setminus K$ of minimal growth in a neighborhood of infinity in $\Gw$.  
	
	Then  for any Hardy-weight $W\gneqq0$ for $L(A,V)$ in $\Gw$, and any compact set $\K$ such that $K\Subset \K\Subset \Gw$, we have $W\in L^1( \K^c, \phi^p \dx)$. 
\end{theorem}
\begin{proof}
	Suppose that $L(A,V)$ is subcritical in $\Gw$, and let $W$ be a Hardy-weight for $L(A,V)$ in $\Gw$. So, 
	\begin{equation}\label{ineq:subcriticality_1}
	Q_{A,V}[\vgf]
	\geq
	\int_\Gw W|\vgf|^p\dx   \qquad \mbox{for all } \vgf\in C_c^\infty(\Gw).
	\end{equation}
	
	Let $K \Subset \Gw$ be a compact set of positive measure with smooth boundary,  and let $\phi\in W_{\rm loc}^{1,p}(\Gw\setminus K)$ be a positive solution of the equation $L(A,V)[u]=0$ in  $\Gw\setminus K$ of minimal growth in a neighborhood of infinity in $\Gw$.

	Let $V_c\in C_c^\infty(K)$ be a nonnegative potential such that $L(A,V-V_c)$ is critical in $\Gw$ with a ground state $\psi$ \cite[Proposition~4.19]{pp}. Since both $\psi$ and $\phi$ are 
	positive solutions of the equation $L(A,V)[u]=0$ in  $\Gw\setminus K$ of minimal growth in a neighborhood of infinity in $\Gw$ \cite[Theorem~5.9]{pp}, it follows that $\psi\asymp \phi$ in $\K^c$.  
	On the other hand, in light of \eqref{ineq:subcriticality_1}, we have 
	\begin{equation} \label{ineq-hardy2}
	Q_{A,V-V_c}[\vgf] =Q_{A,V}[\vgf]=  \int_{K^c}\big( |\nabla \vgf|_A^p + V |\vgf|^p\big)\dx    \geq  \int_{K^c} W\, |\vgf|^p\dx \qquad \forall \vgf\in C_c^\infty(K^c)
	\end{equation} 
	Consequently, Theorem~\ref{prop-L1} implies that $W\in L^1( \K^c, \psi^p \dx)=L^1( \K^c, \phi^p \dx)$. 
\end{proof}

%\iffalse

\begin{remark} \label{rem-w}{\em 
The conditions on the decay of $W$ at infinity given by Theorem \ref{prop-L2} could be compared with the behaviour at infinity of the optimal Hardy-weight given by
Theorem \ref{thm_opt_hardym}. For example, let $V=0$ and $A = \id$ so that $L=-\Delta_p$ and assume that
 $L$ is subcritical in $\Omega$. Let $W$ be an optimal Hardy-weight for $L$, and let $v$ be the ground state of the critical operator $L(\id, -W)$. 
 Then by the null-criticality with respect to $W$
\begin{equation} \label{eq-dp}
\int_{\Omega\setminus K} W(x)\, v^p(x)\, \dx = \infty
\end{equation} 
holds for all $K \Subset \Omega$. This is of course not in contradiction with Theorem  \ref{prop-L2} because $v$ is larger than the function $\phi$ considered in Theorem ~\ref{prop-L2}. For example,  if $v$ is given by the supersolution construction as in Theorem~\ref{thm_opt_hardym} (or \cite[Theorem~1.5]{dfp}), then $v=\gf^{(p-1)/p}$ and by \eqref{eq-dp} we have $W\not \in L^1( \K^c, \phi^{p-1} \dx)$, while by Theorem~\ref{prop-L2}, $W\in L^1( \K^c, \phi^p \dx)$.}
\end{remark}

\medskip

\noindent As a straightforward consequence of Theorem~\ref{prop-L1}, we have:
\begin{corollary}\label{Idan1}
Suppose that $L(A,V)$ of the form \eqref{L-op} is subcritical in $\Gw$ with $A$ and $V$ satisfying either hypothesis {\em(H0)} if $p\geq2$, or {\em (H1)} if $1<p<2$. Let $W\gneqq0$ be a Hardy-weight for $L(A,V)$ such that \eqref{ineq:subcriticality_1} holds true. If $W$ is null-critical, then $W$ is also optimal at infinity in the sense of Definition \ref{def_opt}.
\end{corollary}
\begin{proof}
Let $\phi \in W^{1,p}_{\rm loc}(\Omega)$ be the ground state of the critical operator $L(A,V-W)$ in $\Omega$. Assume that $W$ is not optimal at infinity. 
Then there exists a neighborhood of infinity in $\Omega$, which we denote by $\Omega_\infty$, and a constant $\lambda>1$ such that 
\begin{equation}
Q_{A,V}[\vgf] \ \geq  \ \lambda \int_\Gw W\, |\vgf|^p\dx   \qquad\forall\, \vgf\in C_c^\infty(\Omega_\infty).
\end{equation}
Let $K\Subset\Omega$ be a compact set such that $K^c \subset \Omega_\infty$. Then 
$$
Q_{A,V-W}[\vgf] \ \geq  \  (\lambda-1) \int_\Gw W\, |\vgf|^p\dx   \qquad\forall\, \vgf\in C_c^\infty(K^c).
$$
Hence by Theorem~\ref{prop-L1} we have $W\in L^1( \K^c, \phi^p \dx)$ for any compact set $\K\Subset\Omega$ with $K\Subset \K$. Since $W\in L^1( \K, \phi^p \dx)$ by density and inequality \eqref{ineq:subcriticality_1}, it follows that $W\in L^1(\Omega, \phi^p \dx)$ which is in contradiction with the null-criticality of $W$.
\end{proof}

%%%%%%%%%%%%%%%%%%%%%%%%%%%%%%%%%%%%%%%%%%%%%%%%%%%%%%%%%%%%
\subsection{The linear critical case}\label{sec-cr-lin}
We have
\begin{theorem} \label{thm-nonsym}
Let $P$ be an elliptic operator of the form \eqref{operator} (or  \eqref{op-P}), and assume that $P$ is critical in $\Omega$. Let $\phi$ and $\phi^*$ be the ground states, in $\Omega$,  of $P$ and $P^*$, respectively. 
 Let $K \Subset \Gw$ be a compact set of positive measure with smooth boundary.

Let $W\gneqq0$  be a Hardy-weight for $P$ in $K^c$. Then for any $\K$ such that $K\Subset \K\Subset \Gw$ we have $W\in L^1( \K^c, \phi\,\phi^* \dx)$. 
\end{theorem}

\begin{proof}
The operator $P$ is subcritical in $K^c$. 
Let $G_P^{K^c}(x,y)$ be the minimal positive Green function of $P$ in $K^c$, and let $W\gneqq0$ be a Hardy-weight for $P$ in $K^c$. Then by \cite[Lemma 3.1]{p99} 
\begin{equation} \label{eq-p99}
\int_{K^c} G_P^{K^c}(x,z)\, W(z)\, G_P^{K^c}(z,y)\, \dz\, < \, \infty \qquad \forall\, x,y \in K^c. 
\end{equation}
On the other hand, by Remark~\ref{rem-min-gr}, for any compact $\K$ such that $K\Subset \K\Subset \Gw$, we have for fixed  $x,y \in K^c\setminus \K^c$  
$$
G_P^{K^c}(x,\cdot) \asymp \phi^* \,  ,\qquad   G_P^{K^c}(\cdot,y) \asymp \phi  \qquad \mbox{in } \K^c.
$$
Hence, it follows from \eqref{eq-p99} that for fixed  $x,y \in K^c\setminus \K^c$ it holds
\begin{equation*}
\int_{\K^c} W(z)\, \phi(z)\, \phi^*(z) \dz  \leq  C \int_{\K^c} G_P^{K^c}(x,z)\, W(z)\, G_P^{K^c}(z,y) \dz < \infty. \qedhere
\end{equation*}
\end{proof}

%%%%%%%%%%%%%%%%%%%%%%%%%%%%%%%%%%%%%%%%%%%%%%%%%%%%%%%%%
\medskip

\subsection{The linear subcritical case}\label{sec-scr-lin}
We have
\begin{theorem} \label{prop-P2}
	Consider a linear operator $P$ of the form \eqref{operator} (or  \eqref{op-P}),  satisfying the corresponding local regularity assumption mentioned in Subsection~\ref{sec-lin}. 
	Assume that $P$ is subcritical in $\Gw$. 
	Let $K \Subset \Gw$ be a compact set of positive measure with smooth boundary,  and let $\phi, \phi^*\in W_{\rm loc}^{1,2}(\Gw\setminus K)$ be positive solutions of the equation $Pu=0$ and respectively, $P^*u^*=0$ in $\Gw\setminus K$ of minimal growth in a neighborhood of infinity in $\Gw$.  
	
	Then  for any Hardy-weight $W\gneqq0$ for $P$ in $\Gw$, and any compact set $\K$ such that $K\Subset \K\Subset \Gw$, we have $W\in L^1( \K^c, \phi \phi^* \dx)$. 
\end{theorem}
\begin{proof}
The proof is similar to the proof of Theorem~\ref{prop-L2}, and therefore it is omitted. 
\end{proof}

\medskip

\noindent Similar to the quasilinear case, we have the following consequence of Theorem~\ref{thm-nonsym}.  
\begin{corollary}\label{Idan2}
	Let $P$ be a subcritical elliptic operator of the form \eqref{operator} (or  \eqref{op-P}), and let $W\gneqq0$ be a Hardy-weight for $P$ in $\Gw$. If $P-W$ is null-critical with respect to $W$ in $\Gw$, then $W$ is also optimal at infinity in the sense of Definition~\ref{def-opt}.
\end{corollary}
\begin{proof}
	Let $\phi$ and $\phi^*$ be the ground states, in $\Omega$,  of $P-W$ and $P^*-W$, respectively. 
	Let $K \Subset \Gw$ be a compact set of positive measure with smooth boundary, and assume to the contrary that for some $\gl>1$ the operator $P-\gl W$ is nonnegative in $\Gw\setminus K$. In other words, $(\gl-1)W$ is a Hardy-weight for  $P-W$ in $\Gw\setminus K$.  Then it follows from Theorem~\ref{thm-nonsym} that  for any $\K$ such that $K\Subset \K\Subset \Gw$ we have $W\in L^1( \K^c, \phi\,\phi^* \dx)$, but this contradicts the null-criticality of $P-W$ with respect to $W$. 
	\end{proof}

\begin{remark}
	{\em 
The integrability conditions in the present section are only necessary conditions for $W$ to be a Hardy-weight, as the following elementary example demonstrates.

Consider the operator $P:=-\Gd$ on $\R^N$, where $N\geq 2$ and the weight $W:=1$. Then a positive harmonic function $u$ of minimal growth at infinity in $\R^N$ satisfies 
$u(x)\asymp |x|^{2-N}$, which is in $L^2(\overline{B_R(0)}^c)$ for $R$ sufficiently large if and only if $N\geq 5$. On the other hand, in {\em any} dimension, $W=1$ is not a Hardy-weight for the Laplacian in $\overline{B_R(0)}^c$. 

For sufficient conditions, as well as necessary and sufficient conditions for $W$ to be a Hardy-weight in the {\em linear} case, see \cite{p99} and references therein. }
\end{remark}

\begin{remark}
{\em 
Since criticality theory and in particular, the results concerning Hardy-type inequalities are valid in the setting of second-order elliptic operators on noncompact Riemannian manifolds \cite{dfp,dp,p07}, it follows that the results of the paper hold true also when $\Gw$ is a noncompact Riemannian manifold of dimension $N\geq 2$. }
\end{remark}

%\fi

%%%%%%%%%%%%%%%%%%%%%%%%%%%%%%%%%%%%%%%%%%%%%%%%%%%%%%%%%

\section{\bf Examples} 
\label{s-examples}
\subsection{Example 1}\label{ex1} Consider the case  $\Gw =\R^N$ and $L(A,V) = -\Delta_p$, i.e. the $p$-Laplacian. So, $A= \id$ is the identity matrix, $V=0$ and assume that  $p\geq N\geq 2$. It is well-known that $ -\Delta_p$ is critical in $\R^N$ if and only if $p\geq N$, and therefore, $\phi=1$ is its ground state. 

\smallskip

\noindent We first consider the example $p= N$ and $K:= \overline{B_R(0)}$. Theorem~\ref{thm_opt_hardym} with  $\mathcal{G}(x):=\log \big(|x| /R\big)$ directly implies
the following Leray-type inequality (cf. \cite[Example~13.1]{dfp}).
\begin{lemma} \label{lem-hardy}
Let $K = \overline{B_R(0)} \subset \R^N$ and let $p=N>1$. Then for any $\vgf \in C_c^\infty(K^c)$ it holds 
\begin{equation} \label{hardy-lp}
\int_{K^c} |\nabla \vgf|^N\dx  \geq \left(\frac{N-1}{N}\right)^N \int_{K^c} \frac{|\vgf|^N}{ |x|^N \big(\log (|x| /R) \big)^N}\dx .
\end{equation}
Moreover, $$
W(x) := \left(\frac{N-1}{N}\right)^N\frac{1}{ |x|^N  \big(\log (|x| /R) \big)^N}
$$
is an optimal Hardy-weight for $-\Gd_p$ in $K^c$ with  $v(x)\!:=\!\log^{\frac{N-1}{N}}\!(|x| /R)$ being the ground state for the operator $L(\id, -W)$ in $K^c$.
\end{lemma}

\begin{remark}{\em 
Inequality \eqref{hardy-lp} is well-known (see e.g.~\cite[Thm.~1.14]{ok}), while the optimality of $W$ follows from Theorem~\ref{thm_opt_hardym}. Note also that in the linear case $p=2$ inequality \eqref{hardy-lp} can be generalized, with suitable modifications, to Laplace operators with Robin boundary conditions on $\partial K$, see \cite{km}.
}
\end{remark}

\begin{comment}
\begin{proof}
Without loss of generality we may assume that $R=1$. Since $K^c$ is rotationally symmetric, it suffices to prove that 
\begin{equation} \label{enough}
\int_1^\infty |f'|^p\, r^{p-1}\dr  \ \geq \ \left(\frac{p-1}{p}\right)^p \int_1^\infty\frac{f^p\,  }{ r\,  (\log r)^p }\dr
\end{equation}
holds true for all $f\in C_c^\infty(1,\infty)$. Integration by parts and H\"older inequality
give
\begin{align*}
 \int_1^\infty\frac{f^p}{ r\,  \log^p r }\dr &  =  \frac{p}{p-1} \int_1^\infty \frac{f^{p-1}\, f' }{( \log r)^{p-1}}\dr =  \frac{p}{p-1} \int_1^\infty \frac{f^{p-1} r^{\frac{1-p}{p}}}{( \log r)^{p-1}}\ f' \, r^{\frac{p-1}{p}}\dr \\
 & \quad \leq  \frac{p}{p-1}  \left(\int_1^\infty \frac{f^p }{ r\,  (\log r)^p }\dr \right)^{\frac{p-1}{p}}\   \left(\int_1^\infty |f'|^p\, r^{p-1}\dr\right)^{\frac 1p} \, .
\end{align*}
Hence \eqref{enough}.
\end{proof}

\smallskip

\noindent Inequality \eqref{enough} is not new. It follows, for example, from a more general result given in~\cite[Thm.~1.14]{ok}. Here, for the sake of completeness, we gave a short proof suitable to our setting.
\end{comment}

\noindent
Here the optimal Hardy-weight $W$ is not in $L^1(K^c,\!\dx)$, but Theorem~\ref{prop-L1} implies that $W\in L^1(\K^c,\!\dx)$ for any $\K$ such that $K\Subset \K\Subset \R^N$. 
Indeed, for  any $\rho>R$  such that $\overline{B_\rho(0)} \Subset \K$ we have
$$ 
\int_{\K^c} W \dx \leq \int_{B_\rho(0)^c} \frac{1}{ |x|^N\big(\log (|x| /R) \big)^N}\, \dx \ \asymp\  \int_{\rho/R}^\infty \frac{\dr}{r (\log r )^N} \, < \, \infty.
$$

\noindent Theorem~\ref{prop-L1} now implies that the logarithmic factor in \eqref{hardy-lp} cannot be removed. This is in contrast with the well-known (optimal) Hardy inequality 
\begin{equation} \label{hardy-np}
\int_{\R^N} |\nabla \vgf|^p\dx \ \geq \ \left|\frac{N-p}{p}\right|^p \int_{\R^N} \frac{|\vgf|^p}{ |x|^p}\dx \qquad \forall \vgf\in C_c^\infty(\R^N\setminus\{0\}),
\end{equation}
which holds for $N\neq p$ (see for example, \cite[Example~4.7]{dp}).  Note also that by the optimality of $W$ it follows that the weight function $W$
satisfies  $\int_{\K^c} |v|^{N} \, W\dx=\infty$, where $v(x)=\log^{\frac{N-1}{N}}\!(|x| /R)$, c.f. \eqref{eq-dp}. Indeed, we have
\begin{equation} 
 \int_{\K^c} |v|^{N} \, W\dx= \left(\frac{N-1}{N}\right)^N \int_{\K^c} \frac{1}{ |x|^N \log (|x| /R)}\dx\ \asymp\ \int_{\delta}^\infty \frac{\dr}{r\log r} =\infty,
\end{equation} 
for some $\delta>1$.  

\medskip

\noindent In the case $p>N$, we apply Theorem~\ref{thm_opt_hardym} with $\mathcal{G}(x):=|x|^{(p-N)/(p-1)}$. This gives
$$
W(x):=\left(\frac{p-N}{p}\right)^p\frac{1}{ |x|^p  \big(1- (R/|x|)^{(p-N)/(p-1)} \big)^p}\,.
$$
%and
%$$
%v(x):=\left(|x|^{(p-N)/(p-1)}-R^{(p-N)/(p-1)}\right)^{\frac{p-1}{p}}. 
%$$
In agreement with Theorem~\ref{prop-L1}, we have 
$$
\int_{\K^c} W \dx \ <\,  \infty,
$$
for any $\K$ such that $\overline{B_R(0)} \Subset \K$.

%%%%%%%%%%%%%%%%%%%%%
\subsection{Example 2}\label{ex2} Let $\Gw= \R^N\setminus\{0\}$, $2\leq p=(N+1)/2$, $A= \id$, and 
$$
V(x) = -C_{p,N} \frac{1}{|x|^p}\,, \quad \mbox{where } \ \  C_{p,N}: = \left|\frac{N-p}{p}\right|^p =C_{p,1} \quad \mbox{ is the Hardy constant}.
$$
Here we have 
$$L(A,V)u=Lu= -\Delta_p u +V |u|^{p-2}u=-\Delta_p u - \left|\frac{N-p}{p}\right|^p \frac{|u|^{p-2}u}{|x|^p}\,,$$ 
and the optimality of the integral weight in \eqref{hardy-np} implies in particular, that the operator $L$ is critical in $\Gw$. A straightforward calculation shows that 
$\phi(x) = |x|^{(p-N)/p}$ is the ground state for $L$ in $\Gw$. The null-criticality, implies that  
\begin{equation*} 
\int_{\K^c} |\phi|^{p} \,|V|\dx= C_{p,N}\int_{\K^c} \frac{|x|^{p-N}}{ |x|^p }\dx\ \asymp\ \int_{\delta}^\infty \frac{\dr}{r} =\infty.
\end{equation*} 

\noindent On the other hand, if we set $K = \overline{B_R(0)}$, then 
\begin{equation} \label{hardy-r3}
\int_{K^c} |\nabla \vgf|^p\dx  \geq  C_{p,1} \int_{K^c} \frac{|\vgf|^p}{ (|x|-R)^p}\dx  \qquad \forall \vgf\in C_c^\infty(K^c),
\end{equation}
where the constant $C_{p,1}$ is optimal, see e.g.,~\cite[Section~4]{mmp} and \cite{lp}. Since $C_{p,N} =C_{p,1}$, it follows that
$$
\int_{K^c} \left( |\nabla \vgf|^p - C_{p,1} \frac{|\vgf|^p}{ |x|^p}\right)\dx \geq \int_{K^c} W(x) \vgf^p\dx \qquad \forall \vgf\in C_c^\infty(K^c),
$$
with the Hardy-weight
$$
W(x) := C_{p,1} \left(\frac{1}{ (|x|-R)^p}- \frac{1}{|x|^p}\right)\ > 0 \qquad \mbox{in } K^c .
$$
It is now easy to verify that 
$$
\int _{\K^c} W|\phi|^p\dx < \infty,
$$
for any $K\Subset \K\Subset\R^N$.
\subsection{Example 3}\label{ex3} Let $L$, $\Gw$, and $K$ be as in Example~2 with $p=2$ and $N\geq3$. Let $\phi(x) = |x|^{(2-N)/2}$ and $\psi(|x|):=\phi(|x|) \log(|x|/R)$.
Using the supersolution construction with $\phi$ and $\psi$ in $K^c$, we obtain the following optimal Hardy inequality in $K^c$
$$
\int_{K^c} \left( |\nabla \vgf|^2 - \left(\frac{N-2}{2}\right)^2\frac{|\vgf|^2}{ |x|^2}\right)\dx \geq \int_{K^c} W|\vgf|^2\dx \qquad \forall \vgf\in C_c^\infty(K^c),
$$
with ground state $\phi_K:=(\phi\psi)^{1/2}$, and
$$
W(x) := \frac{1}{ 4|x|^2  \big(\log (|x| /R) \big)^2}\ > 0.
$$ 
Recall that $\phi$ is a ground state of the critical operator $L$ in $\Gw$. It is now easy to verify that as claimed in Theorem~\ref{prop-L1}  for any $K\Subset \K\Subset\R^N$ and any $\rho >R$ such that $B_\rho(0) \Subset \K$ we have 
$$
\int _{\K^c} W|\phi|^2\dx\ \asymp\ \int_{\delta}^\infty \frac{\dr}{r(\log r)^2} < \infty,
$$
 where $\delta = \rho/R >1$.  On the other hand, the optimality of $W$  in $K^c$ implies that  there exists $\delta>1$ such that
$$
\int _{\K^c} W|\phi_K|^2\dx=\int _{\K^c} W|\phi|^2\log (|x| /R)\dx\ \asymp  \ \int_{ \delta}^\infty \frac{\dr}{r\log r} = \infty,
$$
demonstrating the sharpness of Theorem~\ref{prop-L1}.

\subsection{Example 4} \label{ex4}
Let $\Gw =\R^N$, $A= \id$ the identity matrix, $V=0$ and $p<N$. In this case we have that $L(A,V) = -\Delta_p$ is subcritical operator in $\R^N$. The fundamental solution $|x|^{(p-N)/(p-1)}$ with a singular point at the origin  
is a positive $p$-harmonic function of minimal growth at infinity in $\R^N$.
Assume further that $N\geq 3$. Then the following Hardy inequality holds true with a subcritical Hardy-weight
\begin{equation} \label{hardy-np4}
\int_{\R^N} |\nabla \vgf|^p\dx \ \geq \ \left|\frac{N-p}{p}\right|^p \int_{\R^N} \frac{|\vgf|^p}{ 1+|x|^p}\dx \qquad \forall \vgf\in C_c^\infty(\R^N).
\end{equation}
By Theorem~\ref{prop-L2}, we have 
$$\int_{|x|>R} \frac{\Big(|x|^{(p-N)/(p-1)}\Big)^p}{ 1+|x|^p}\dx\ \asymp\  \int_R^\infty r^{(1-N)/(p-1)}\dr<\infty.$$
The optimality at infinity of the Hardy-weight $C_{p,N}|x|^{-p}$ in $\R^N\setminus \{0\}$ with a ground state $|x|^{(p-N)/p}$ implies, on the other hand, that
 $$
 \int_{|x|>R} \frac{\Big(|x|^{(p-N)/p}\Big)^p}{|x|^p}\dx \ \asymp\  \int_R^\infty r^{-1}\dr=\infty.
 $$

\subsection{Example 5}(cf. \cite{bft,bft1,d})\label{ex5} 
Let $\Gw$ be a $C^{1,\ga}$-bounded domain in $\R^N$. and let $A= \id$ be the identity matrix, and $p=2$. Fix $x_0\in \Gw$, and denote $G(x):=G_{-\Gd}^\Gw(x,x_0)$. Let $K \Subset \Gw$ be a compact set of positive measure with smooth boundary and suppose that $K$ is large enough such that $G(x) <1$ in $K^c$. For $t \in (0,1)$ 
define 
$$
X_0(t) = 1,  \quad X_1(t) = (1-\log(t))^{-1},
$$
and
$$
X_{i+1}(t) = X_1(X_i(t)), \qquad \forall\, i\geq 1. 
$$
For $i\geq 0$, and $x\in K^c$, let 
$$
W_i(x):=\frac{|\nabla G(x)|^2}{4G^{2}(x)}\ \sum_{k=0}^i \ \prod_{j=0}^k X^2_j(G(x)) ,\qquad u_i(x):= \left(G(x)\prod_{j=0}^i X_j(G(x))^{-1}\right)^{1/2},
$$ and let $W_i=0$ in $K$.     
Then $W_i$ is a Hardy-weight for $-\Gd$ in $\Gw$ (obtained by the supersolution construction), and $u_{i}$ is a positive solution of the equation $$(-\Gd-W_i)u=0$$ 
of minimal growth in a neighborhood of  infinity in $\Gw$. Let
$$
R_i(x)=W_{i+1}(x)-W_{i}(x) = \frac{|\nabla G(x)|^2}{4G^{2}(x)}\ \prod_{j=0}^{i+1} X^2_j(G(x)).
$$
Then $R_i$ is a Hardy-weight for $-\Delta-W_i$ in $K^c$ and a straightforward calculation shows that 
\begin{align*}
\int_{K^c}u_i^2(x)\, R_i(x) \dx & =  \int_{K^c}  \frac{|\nabla G(x)|^2}{4G(x)} \,  \prod_{j=0}^{i} X_j(G(x))\,  X^2_{i+1}(G(x))\, \dx \\
& \, \asymp\, \int_0^1  \prod_{j=0}^{i} X_j(t)\,  X^2_{i+1}(t)\, \frac{\dt}{t}   \asymp\, \int_1^\infty \frac{1}{(1+\log t)^2} \, \frac{\dt}{t} \ < \  \infty,
\end{align*}
in agreement with Theorem \ref{prop-L2}. 
\subsection{Example 6}(cf. \cite{lp})\label{ex6}
 Let $\Gw$ be a $C^{1,\ga}$-bounded domain in $\R^N$, and fix $1<p<\infty$. Let $\gd_\Gw:\Gw\to \R_+$ be the distance function to $\partial \Gw$. Then there exists $\Gw'\subset \Gw$, a neighborhood of infinity in $\Gw$ such that  the following Hardy inequality holds  
 \begin{equation} \label{hardy-r5}
 \int_{\Gw'} |\nabla \vgf|^p\dx  \geq  C_{p,1} \int_{\Gw'} \frac{|\vgf|^p}{\gd_\Gw^p}\dx  \qquad \forall \vgf\in C_c^\infty(\Gw').
 \end{equation}
 A positive solution $u$ for $L(\id,-C_{p,1}\gd_\Gw^{-p})$ of minimal growth at infinity in $\Gw$ behaves like $\gd_\Gw^{(p-1)/p}$ (see  \cite{lp}). Hence, for $W:=C_{p,1}\gd_\Gw^{-p}$ we obtain
 $\int_{\Gw'}  W\left(\gd_\Gw^{(p-1)/p}\right)^p\dx=\infty$.
In particular, by Theorem~\ref{prop-L2} we have
$$
\lambda_\infty(-\Gd_p, \gd_\Gw^{-p},\Omega):=\sup\big \{\gl\in\Real\mid \exists K\Subset
\Gw\mbox{ s.t. } -\Gd_p-\gl \gd_\Gw^{-p}\geq 0 \; \mbox{ in }  \Gw\setminus K \big\}= C_{p,1}. 
$$

% \medskip
 
On the other hand, let $\gl<C_{p,1}$, then $W_\gl:=(C_{p,1}-\gl)(\gd_\Gw)^{-p}$ is a Hardy-weight for the subcritical operator $L_\gl:=L(\id,-\gl \gd_\Gw^{-p})$ in $\Gw'$.    
Any positive solution $v$ for $L_\gl$ of minimal growth at infinity in $\Gw$ behaves like $\gd_\Gw^{\ga(\gl)}$, where $\ga(\gl)>(p-1)/p$ is a solution of the transcendental equation
$\lambda=(p-1)\alpha ^{p-1}(1-\alpha)$, see \cite{lp}. Note that $\ga(\gl)\to (p-1)/p$ as $\gl\to C_{p,1}$. Hence, we have
  $\int_{\Gw'}  W_\gl\left(\gd_\Gw^{\ga(\gl)}\right)^p\dx<\infty$, demonstrating again the sharpness of Theorem~\ref{prop-L2}.

%%%%%%%%%%%%%%%%%%%%%%%%%%%%%%%%%%%%%%%%%%%%%%%%%%%%%%%%%

\bigskip

%\newpage
%%%%%%%%%%%%%%%%%%%%
\begin{center}{\bf Acknowledgments} \end{center}
The authors are grateful to Idan Versano for pointing out to them corollaries~\ref{Idan1} and \ref{Idan2}, and to Georgios Psaradakis for a useful discussion.
The authors thanks the departments of Mathematics at the Technion and at the Universit\`{a} degli studi di Brescia for the hospitality during their mutual visits Y.~P. acknowledges the support of the Israel Science Foundation (grant 970/15) founded by the Israel Academy of Sciences and Humanities. H.~K.  acknowledges the support of the Project FFABR of the Italian Ministry of Education.
%%%%%%%%%%%%%%%%%%%%%%%%%%%%%%
%%%%%%%%%%%%%%%%%%%%%%%%%%%%%

\end{document}